\documentclass[pdflatex,sn-mathphys-num]{sn-jnl}

\usepackage{graphicx}%
\usepackage{multirow}%
\usepackage{amsmath,amssymb,amsfonts}%
\usepackage{amsthm}%
\usepackage{mathrsfs}%
\usepackage[title]{appendix}%
\usepackage{xcolor}%
\usepackage{textcomp}%
\usepackage{manyfoot}%
\usepackage{booktabs}%
\usepackage{algorithm}%
\usepackage{algorithmicx}%
\usepackage{algpseudocode}%
\usepackage{listings}%
\theoremstyle{thmstyleone}%
\newtheorem{thm}{Theorem}%
\newtheorem{prp}[thm]{Proposition}%
\newtheorem{crl}[thm]{Corollary}%
\theoremstyle{thmstyletwo}%
\newtheorem{exm}{Example}%
\theoremstyle{thmstylethree}%
\newtheorem{dfn}{Definition}%

\usepackage{subcaption}
\usepackage{here}
\usepackage{physics}
\usepackage{mathtools}
\newcommand{\e}[1]{\begin{equation}#1\end{equation}}
\newcommand{\ald}[1]{\begin{aligned}#1\end{aligned}}
\newcommand{\ite}[1]{\begin{itemize}#1\end{itemize}}
\newcommand{\prn}[1]{\left(#1\right)}
\newcommand{\ang}[1]{\left\langle#1\right\rangle}
\newcommand{\cur}[1]{\left\{#1\right\}}
\newcommand{\squ}[1]{\left[#1\right]}

\def\r{\mathbb{R}}
\def\rn{\mathbb{R}^n}
\allowdisplaybreaks

\begin{document}

\title[Article Title]{Revisiting Invex Functions: Explicit Kernel Constructions and Characterizations}

\author*[1]{\fnm{Akatsuki} \sur{Nishioka}}\email{nishioka.a.2122@m.isct.ac.jp}

\affil*[1]{\orgdiv{Department of Mathematical and Computing Science}, \orgname{School of Computing, Institute of Science Tokyo}, \orgaddress{\street{Ookayama 2-12-1}, \city{Meguro-ku}, \postcode{152-8550}, \state{Tokyo}, \country{Japan}}}

\abstract{
An invex function generalizes a convex function in the sense that every stationary point is a global minimizer. Recently, invex functions and their subclasses have attracted attention in signal processing and machine learning. However, verifying invexity is often difficult because its definition involves an unknown function called a kernel function. This paper studies kernel functions associated with invex functions, which have received relatively limited attention in the literature. In particular, we develop several methods for constructing explicit kernel functions and establish a characterization of pseudoconvexity in terms of kernel functions. These results provide constructive tools for proving invexity of new functions and for clarifying their structural properties. We also present examples of nonsmooth, non-pseudoconvex invex functions arising in signal processing.
}



\maketitle

\section{Introduction}

A convex function has the property that every stationary point is a global minimizer. This property extends to a broader class of functions known as invex functions. However, verifying invexity is generally difficult, as its definition involves an unknown kernel function.

Knowledge of a kernel function is important for practical applications. For example, the sum of invex functions is invex if they are invex with a common kernel function, and a constrained optimization problem is invex (every Karush--Kuhn--Tucker point is globally optimal) if the objective and constraint functions are invex with a common kernel function \cite{mishra08}. Moreover, recent optimization algorithms for invex problems utilize kernel functions \cite{barik23}.

Despite their importance, kernel functions for invex functions have received relatively limited attention in the literature, and systematic methods for proving invexity through kernel constructions remain scarce. It has also been pointed out that the theory of invexity is often rather abstract and accompanied by relatively few concrete examples and applications \cite{zalinescu14}. Motivated by applications involving nonsmooth invex functions, this paper studies kernel functions from a constructive viewpoint. Specifically, we develop explicit construction rules for kernel functions and establish a characterization of pseudoconvexity in terms of kernel functions. These results provide a useful framework for analyzing invex functions and some of their subclasses. More broadly, the structure of kernel functions may be used to characterize subclasses and structural properties of invex functions.

\subsection{Related work}

An invex function was introduced in \cite{hanson81} for a differentiable function. It was extended to nonsmooth functions \cite{reiland90} and functions defined on Riemannian manifolds \cite{Pini94}. See the monograph \cite{mishra08} for details. Recently, invex functions and subclasses have attracted attention in machine learning and image processing \cite{barik22,barik23,hinder20,pinilla22,pinilla24}.

There are many subclasses of invex functions. One such example is the class of pseudoconvex functions (see \cite{stancu12}).  The invexity of nonlinear semidefinite programming with relation to pseudoconvex optimization is studied in \cite{nishioka25jogo}. For applications of pseudoconvex optimization in economics, management science, and structural engineering, see \cite{bian18,nishioka23coap,nishioka25nfao,stancu12}. Quasar-convex (star-convex) functions are another subclass of invex functions that have been attracting attention in machine learning \cite{hinder20}.

\subsection{Contributions}

\ite{
\item We clarify the connections between invexity, pseudoconvexity, and quasiconvexity for possibly nonsmooth locally Lipschitz continuous functions, which are important for applications in signal processing. 
\item We establish a characterization of pseudoconvexity in terms of the structure of kernel functions.
\item We develop explicit construction rules for kernel functions of (possibly nonsmooth) invex functions.
\item We present simple and constructive proofs for the invexity of several nonsmooth functions arising in applications.
}

\subsection{Notation and organization}

Throughout the paper, $X\subseteq\rn$ is a nonempty open set and $C\subseteq\rn$ is a nonempty open convex set. $\mathrm{conv}\,X$ is the convex hull of a set $X$. $\|x\|=\|x\|_2\coloneqq\sqrt{\sum_{i=1}^n x_i^2}$ is the $l_2$ norm, $\|x\|_1\coloneqq\sum_{i=1}^n |x_i|$ is the $l_1$ norm, $\ang{x,y}$ is the standard inner product for $x,y\in\rn$.

This paper is organized as follows. Section 2 provides preliminaries on invex functions. Section 3 provides connections of invex functions to pseudoconvex and quasiconvex functions. The characterization theorem for pseudoconvex functions in terms of kernel functions is also provided. Section 4 provides examples of invex functions and systematic methods to construct their kernel functions. Finally, Section 5 provides concluding remarks.

\section{Preliminaries on invex functions}\label{sec_pre}

We summarize the definitions and known results of invex functions. We first consider a smooth setting, and then consider a nonsmooth setting for clarity. We also illustrate why finding kernel functions is important in applications. We present nonsmooth extensions of basic theorems given in \cite{mishra08} and provide their proof in Appendix \ref{a_proof} for the reader’s convenience.

\subsection{Smooth case}

\begin{dfn}[smooth invex functions (e.g., \cite{mishra08})]\label{dfn_invex}
Let $X\subseteq\rn$ be a nonempty open set. A differentiable function $f:X\to\r$ is said to be invex if there exists a vector-valued function (called a kernel function) $\eta:X\times X\to \rn$ such that
\e{\label{invex_smo}
f(y)-f(x)\ge\ang{\nabla f(x),\eta(x,y)},\ \ \forall x,y\in X.
}
\end{dfn}

A first idea of invex functions was introduced by Hanson \cite{hanson81}, and the term ``invex'' was coined by Craven \cite{craven81} as an abbreviation of ``invariant convex'' (see~Section \ref{sec_inv}). Every convex function is invex with a kernel function $\eta(x,y)=y-x$.

The following fundamental theorem characterizes invex functions.

\begin{prp}[{\cite[Theorem 2.2]{mishra08}}]\label{prp_invex}
Let $f:X\to\r$ be differentiable. $f$ is invex if and only if every stationary point of $f$ (a point $x^*\in X$ satisfying $\nabla f(x^*)=0$) is a global minimizer of $f$.
\end{prp}

\begin{figure}[t]
    \centering
    \includegraphics[width=5cm]{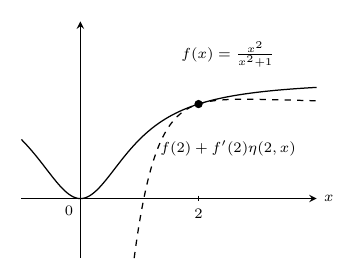}
    \caption{Graphs of an invex function $f(x)=x^2/(x^2+1)$ (solid line) and its tangent curve at $x=2$: $f(2)+f'(2)\eta(2,x)=4/5+4(x-2)/(x^2+1)^2$ (dashed line). The invexity of $f$ and a corresponding kernel function are given by Corollary \ref{crl_frac}.}
    \label{fig_tangent}
\end{figure}

A kernel function $\eta$ is not unique and can be discontinuous. Figure \ref{fig_tangent} illustrates the interpretation of kernel functions. While a convex function $g$ is bounded below by the tangent line $l_x(y)=g(x)+\ang{\nabla g(x),y-x}$ at a point $x$, an invex function $f$ is bounded below by a tangent curve $c_x(y)=f(x)+\ang{\nabla f(x),\eta(x,y)}$ at $x$. Note that we can choose $\eta$ satisfying $\eta(x,x)=0$ (see \eqref{eta_non}). Since the tangent curve always satisfies $c_x(y)\equiv f(x)$ when $x$ is stationary, $x$ must be a global minimizer.

\subsection{Nonsmooth case}

We introduce nonsmooth invex functions using Clarke subdifferentials (called C-invex functions in \cite{mishra08}). There are other generalizations of invex functions to nonsmooth functions \cite{mishra08}, but the following one is common and suitable for our purpose.

\begin{dfn}[Clarke subdifferential and stationarity]\label{dfn_cla}
Let $f:X\to\r$ be locally Lipschitz continuous. The Clarke subdifferential of $f$ at $x\in X$ is defined by
\e{
\partial f(x)=\mathrm{conv}\cur{\xi\in\rn\,\middle|\,\exists \{x^k\}\subseteq D_f\ \mathrm{s.t.}\ \underset{k\to\infty}{\lim}x^k=x,\ \underset{k\to\infty}{\lim}\nabla f(x^k)=\xi},
}
where $D_f\subseteq X$ is the set of points where $f$ is differentiable\footnote{A locally Lipschitz continuous function is differentiable almost everywhere by Rademacher's theorem (see \cite{cui21}).}. A point $x\in X$ satisfying $0\in\partial f(x)$ is called a (Clarke) stationary point.
\end{dfn}

The Clarke subdifferential $\partial f(x)$ at $x$ is a nonempty convex compact subset of $\rn$. If $f$ is differentiable at $x$, then $\partial f(x)=\{\nabla f(x)\}$ holds. If $f$ is convex, then the Clarke subdifferential coincides with the convex subdifferential defined by $\partial f(x)\coloneqq\{\xi\in\rn\mid f(y)-f(x)\ge\ang{\xi,y-x},\ \forall x,y\in\rn\}$. The stationarity condition $0\in\partial f(x)$ is a necessary condition for local minimality. Note that a local maximizer also satisfies $0\in\partial f(x)$. See \cite{clarke90,cui21} for details.

\begin{dfn}[nonsmooth invex functions (e.g., \cite{mishra08})]\label{dfn_invex_non}
A locally Lipschitz continuous function $f:X\to\r$ is said to be invex if there exists a vector-valued function $\eta:X\times X\to\rn$ such that
\e{\label{invex_non}
f(y)-f(x)\ge\ang{\xi,\eta(x,y)},\ \ \forall x,y\in X,\ \forall \xi\in\partial f(x).
}
\end{dfn}

When $f$ is differentiable, Definition \ref{dfn_invex_non} coincides with Definition \ref{dfn_invex}. Hereafter, ``invex'' is used in the sense of Definition \ref{dfn_invex_non}. We state the nonsmooth counterpart of Proposition \ref{prp_invex}. We give a complete proof in Appendix \ref{a_proof}, which is omitted in \cite{mishra08}.

\begin{prp}[{\cite[Theorem 4.33]{mishra08}}]\label{prp_invex_non}
Let $f:X\to\r$ be locally Lipschitz continuous. $f$ is invex if and only if every stationary point (a point $x^*\in X$ satisfying $0\in\partial f(x^*)$) is a global minimizer of $f$.
\end{prp}

\subsection{Importance of kernel functions}

The sum of two invex functions is not necessarily invex. For example, $f_1(x,y)=x-y^2$ and $f_2(x,y)=-x$ are both invex since they have no stationary points, but the sum $f(x,y)=-y^2$ is not invex. The following theorem tells us that if two functions are invex with the same kernel function, then the sum is also invex with that kernel function. See Appendix \ref{a_proof} for a proof.

\begin{prp}[Extension of {\cite[Theorem 2.9]{mishra08}} to the nonsmooth setting]\label{prp_sum}
If $f:X\to\r$ and $g: X\to\r$ are invex with the same kernel function $\eta$, then $\alpha f+\beta g$ is invex with $\eta$ for any $\alpha,\beta\ge0$.
\end{prp}

Also, invexity is not necessarily preserved by constraints. For example, the function $f(x,y)=x-y^2$ is invex, but its minimization under a box constraint $[-1,0]\times[-2,1]$ has a non-global local minimizer $(-1,1)$ while the global minimizer is $(-1,-2)$. The following theorem shows that if the objective and constraint functions are invex with the same kernel function, then the constrained optimization problem is invex. See Appendix \ref{a_proof} for a proof.

\begin{prp}[Extension of {\cite[Section 5.1]{mishra08}} to the nonsmooth case]\label{prp_const}
Consider a constrained optimization problem
\e{\ald{
&\mathrm{minimize} & & f(x)\\
&\mathrm{subject\ to} & & g_i(x)\le0\ (i=1,\ldots,m).
}}
If $f:X\to\r$ and $g_i:X\to\r\ (i=1,\ldots,m)$ are invex with the same kernel function $\eta$, then every point satisfying the Karush--Kuhn--Tucker (KKT) conditions 
\e{\ald{\label{kkt}
& 0\in\partial f(x)+\sum_{i=1}^m\lambda_i\partial g_i(x),\\
& g_i(x)\le0,\ \lambda_i\ge0,\ \lambda_ig_i(x)=0\ (i=1,\ldots,m),
}}
is globally optimal.
\end{prp}

Propositions \ref{prp_sum} and \ref{prp_const} highlight the importance of explicit kernel functions for invex functions. 
Such kernel functions are not only useful for theoretical analysis but also play a role in optimization algorithms for invex problems \cite{barik23}. 
In many cases, invexity is established without constructing kernel functions, for example, by verifying that every stationary point is globally optimal \cite{pinilla22}. 
In contrast, in Section \ref{sec_ex}, we develop systematic methods for proving invexity by explicitly constructing kernel functions. 
Our goal is to obtain kernel functions that are simple, reflect the structure of the underlying function, and avoid case-by-case definitions such as \eqref{eta_non}.

\section{Connections to other generalizations of convexity}\label{sec_rel}

We show the relationship between invex functions, pseudoconvex functions, and quasiconvex functions for possibly nonsmooth locally Lipschitz continuous functions. This generalizes the result in \cite{mishra08} to the nonsmooth setting and is important for treating nonsmooth non-pseudoconvex invex functions that often appear in signal processing.  We summarize the relationship and examples in Figures \ref{fig_venn} and \ref{fig_ex}. Then, we provide a characterization theorem for pseudoconvex functions in terms of kernel functions.

\begin{figure}[t]
    \centering
    \includegraphics[width=7cm]{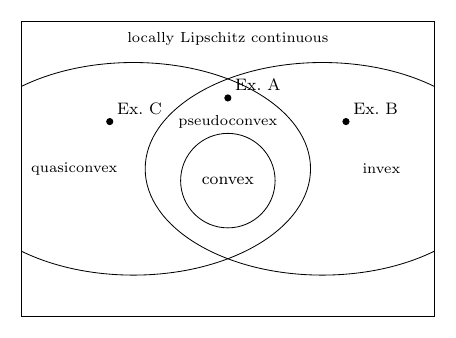}
    \caption{The Venn diagram of convex, pseudoconvex, quasiconvex, and invex functions under the assumption of locally Lipschitz continuity. Ex.~A-C are shown in Figure \ref{fig_ex}. Note that, under the assumption of local Lipschitz continuity, the class of pseudoconvex functions coincides with the intersection of the classes of invex and quasiconvex functions (Theorem \ref{thm_quasi_invex}).}
    \label{fig_venn}
\end{figure}

\begin{figure}[t]
  \centering
  \begin{tabular}{ccc}
  \begin{minipage}[t]{0.33\hsize}
    \centering
    \includegraphics[width=3.5cm]{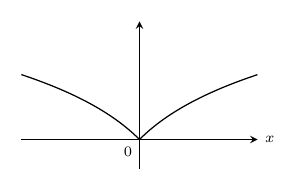}
    \subcaption{Ex.~A (the graph of $f(x)=\log(|x|+1)$)}
  \end{minipage} &
  \hspace{-2mm}
  \begin{minipage}[t]{0.33\hsize}
    \centering
    \includegraphics[width=3cm]{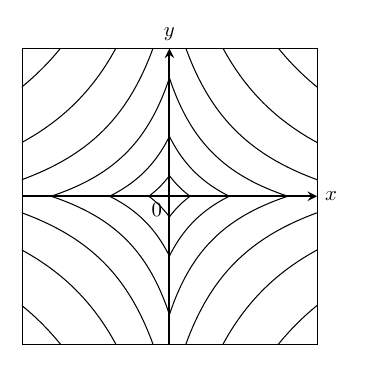}
    \subcaption{Ex.~B (contour lines of $f(x,y)=\log(|x|+1)+\log(|y|+1)$)}
  \end{minipage} &
  \hspace{-2mm}
  \begin{minipage}[t]{0.33\hsize}
    \centering
    \includegraphics[width=3.5cm]{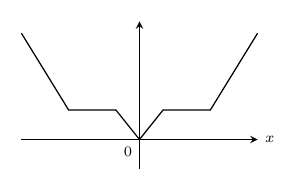}
    \subcaption{Ex.~C}
  \end{minipage}
  \end{tabular}
  \caption{Examples of functions shown in Figure \ref{fig_venn}. (a) is pseudoconvex. (b) is invex but not pseudoconvex (has nonconvex sublevel sets). (c) is quasiconvex but not invex (has stationary points that are not global minimizers).}
  \label{fig_ex}
\end{figure}

\begin{dfn}[pseudoconvex function (e.g., \cite{penot97})]\label{dfn_pseudo_non}
A locally Lipschitz continuous function $f:C\to\r$ is said to be pseudoconvex if
\e{
f(x)>f(y)\ \Rightarrow\ \forall \xi\in\partial f(x),\ \ang{\xi,y-x}<0
\label{pseudo1}
}
or equivalently,
\e{
\exists \xi\in\partial f(x),\ \ang{\xi,y-x}\ge 0\ \Rightarrow\ f(x)\le f(y)
\label{pseudo2}
}
holds for any $x,y\in C$.
\end{dfn}

\begin{dfn}[quasiconvex function (e.g.,\cite{penot97})]
A function $f:C\to\r$ is said to be quasiconvex if its sublevel set $\{x\in C \mid f(x)\le\alpha\}$ is convex for any $\alpha\in\mathbb{R}$.
\end{dfn}

If $f$ is locally Lipschitz continuous, there is an equivalent definition of quasiconvexity:
\e{
f(x)> f(y)\ \Rightarrow\ \forall \xi\in\partial f(x),\ \ang{\xi,y-x}\le 0
\label{quasi2}
}
holds for any $x,y\in C$ \cite[Proposition 3.1]{penot97}. This implies that every pseudoconvex function is quasiconvex.

\begin{thm}[Extension of {\cite[Theorem 2.25]{mishra08}} to the nonsmooth case]
Consider locally Lipschitz continuous functions defined on an open convex set $C\subseteq\rn$. The class of pseudoconvex functions is strictly included in the class of invex functions if $n>1$. If $n=1$, the two classes coincide. 
\end{thm}
\begin{proof}
Every pseudoconvex function is invex since its stationary point is always a global minimizer (substitute $\xi=0$ in \eqref{pseudo2}). When $n>1$, there exist invex functions that are not quasiconvex, hence not pseudoconvex. See Section \ref{sec_non_quasi} for such examples.

We consider the case $n=1$. Let $f:C\to\r$ be invex. We show that level sets $L_f(\alpha)\coloneqq\{x\in C\mid f(x)\le\alpha\}$ are convex for any $\alpha\in\r$. Assume that there exists $\alpha\in\r$ such that the level set $L_f(\alpha)$ is not convex, i.e., it consists of at least two disjoint intervals. Consider two consecutive intervals $I_1,I_2\subset L_f(\alpha)$ such that $I_1$ lies on the left side of $I_2$. By the continuity of $f$, there exist the right endpoint $\bar{x}_1$ of $I_1$ and the left endpoint $\bar{x}_2$ of $I_2$ satisfying $\bar{x}_1<\bar{x}_2$ and $f(\bar{x}_1)=f(\bar{x}_2)=\alpha$. Then, by the mean value theorem for Clarke subdifferentials \cite[Theorem 2.3.7]{clarke90}, there exists $x^*\in(\bar{x}_1,\bar{x}_2)$ such that $0\in\partial f(x^*)$. Since $x^*\notin L_f(\alpha)$ (i.e., $f(x^*)>\alpha$), $x^*$ is not a global minimizer, which contradicts $f$ being invex.
\end{proof}

\begin{thm}[Extension of {\cite[Theorem 2.27]{mishra08}} to the nonsmooth case]\label{thm_quasi_invex}
Consider locally Lipschitz continuous functions defined on an open convex set $C\subseteq\rn$. Under the assumption of quasiconvexity, the classes of pseudoconvex functions and invex functions coincide.
\end{thm}
\begin{proof}
Let $f:C\to\r$ be a quasiconvex function. It suffices to show that if $f$ is invex, then $f$ satisfies the definition of pseudoconvexity \eqref{pseudo2}.

First, if $x^0\in C$ satisfies $0\in\partial f(x^0)$, then, by the invexity, $f(x^0)\le f(x)$ holds for any $x\in C$. Therefore, the definition of pseudoconvexity \eqref{pseudo2} holds.

Next, we consider the case when $x^0$ satisfies $0\notin\partial f(x^0)$ and prove that \eqref{pseudo2} also holds for this case. Suppose, for the sake of contradiction, that there exists $x^1\in C$ and $0 \neq \xi^0\in\partial f(x^0)$ such that
\e{\label{prf_quasi1}
\ang{\xi^0,x^1-x^0}\ge 0,
}
but
\e{\label{prf_quasi2}
f(x^0)>f(x^1).
}
By \eqref{prf_quasi2} and the quasiconvexity \eqref{quasi2}, we have
\e{
\ang{\xi,x^1-x^0}\le 0
}
for any $\xi\in\partial f(x)$. Thus, combined with \eqref{prf_quasi1}, it follows
\e{\label{prf_quasi3}
\ang{\xi^0,x^1-x^0}=0.
}
Note that $H=\{x\in\rn\mid\langle\xi^0,x-x^0\rangle=0\}$ is a supporting hyperplane of a sublevel set $X_0=\{x\in\rn\mid f(x)\le f(x^0)\}$, which is nonempty, closed, and convex due to the continuity and quasiconvexity of $f$. By \eqref{prf_quasi2} and \eqref{prf_quasi3}, $x^1$ lies in the interior of $X_0$ and on its supporting hyperplane, which is a contradiction. Therefore, \eqref{pseudo2} holds for any $x,y\in C$.
\end{proof}

The following theorem characterizes pseudoconvex functions in terms of kernel functions.

\begin{thm}\label{thm_inv_pse}
Let $f:C\to\r$ be invex. $f$ is pseudoconvex if and only if $f$ admits a kernel function of the form $\eta(x,y)=\alpha(x,y)(y-x)$ with $\alpha(x,y)\ge0$.
\end{thm}
\begin{proof}
Assume that $f$ admits a kernel function of the form $\eta(x,y)=\alpha(x,y)(y-x)$ with $\alpha(x,y)\ge0$ and show that \eqref{pseudo2} holds. If there exists $\xi\in\partial f(x)$ such that $\ang{\xi,y-x}\ge0$, then by the invexity and $\alpha(x,y)\ge0$, we obtain
\e{
f(y)-f(x) \ge\ang{\xi,\eta(x,y)}=\alpha(x,y)\ang{\xi,y-x} \ge 0
}
Thus, $f$ is pseudoconvex.

Conversely, we set 
\e{
\eta(x,y)=
\begin{cases}
\frac{(f(y)-f(x))}{\ang{\xi^*,y-x}}(y-x) & \mathrm{if}\ f(x)>f(y),\\
0 & \mathrm{otherwise},
\end{cases}
}
where we define $\xi^*\in\partial f(x)$, for any $x,y\in C$, by $\xi^*=\mathrm{arg\,min}_{\xi\in\partial f(x)}\,\ang{\xi,x-y}$ so that the necessary optimality condition yields $\ang{\xi-\xi^*,y-x}\le0$, hence $\ang{\xi,y-x}/\ang{\xi^*,y-x}\ge1$, for any $\xi\in\partial f(x)$. Note that, by the pseudoconvexity \eqref{pseudo1}, $\ang{\xi^*,y-x}<0$ for any $x,y\in C$ satisfying $f(x)>f(y)$. The above $\eta$ is the kernel function of $f$ of the form $\alpha(x,y)(y-x)$ with $\alpha(x,y)\ge0$.
\end{proof}

In the following section, we study kernel functions of several pseudoconvex and non-pseudoconvex invex functions.

In other generalizations of convexity, quasar-convex functions and functions satisfying the Polyak--{\L}ojasiewicz inequality are known to be invex (see \cite{barik23,hinder20}).

\section{Explicit kernel constructions for invex functions}\label{sec_ex}

\subsection{Fractional programming: pseudoconvex example 1}

Fractional programming is a class of problems that can be written as follows:
\e{
\underset{x}{\mathrm{minimize}}\ \frac{g(x)}{h(x)}.
}
Under certain assumptions (e.g., $g$ is convex, nonnegative, and $h$ is concave, positive), this problem becomes a pseudoconvex optimization problem \cite{stancu12}. Such problems naturally arise in contexts where one optimizes a ratio. Applications of fractional programming and pseudoconvex optimization in economics and management science are presented in \cite{stancu12}.

\begin{thm}\label{thm_frac}
Let $f:C\to\r$ be a function defined by
\e{
f(x)=g(x)/h(x)
}
where $g:C\to\r$ is convex and $g(x)\ge0$ for any $x\in C$ and $h:C\to\r$ is concave and $h(x)>0$ for any $x\in C$. Then, $f$ is invex with a kernel function 
\e{
\eta(x,y)=\frac{h(x)}{h(y)}(y-x).
}
\end{thm}
\begin{proof}
By the quotient rule of Clarke subdifferentials \cite[Proposition 2.3.14]{clarke90}, we have
\e{
\partial f(x)=\frac{1}{h(x)^2}(h(x)\partial g(x)-g(x)\partial h(x))=\frac{1}{h(x)}(\partial g(x)-f(x)\partial h(x)).
}
By the convexity of $g$, and the concavity of $h$, we obtain
\e{
\ald{
\ang{\xi_f,\eta(x,y)}
&= \frac{1}{h(y)}\ang{\xi_g-f(x)\xi_h,y-x}\\
&\le \frac{1}{h(y)}(g(y)-g(x)-f(x)(h(y)-h(x)))\\
&= f(y)-f(x),
}
}
for any $\xi_f\in\partial f(x)$.
\end{proof}

It is known that the function defined in Theorem \ref{thm_frac} is actually pseudoconvex \cite{stancu12}. Indeed, it is a direct consequence of Theorem \ref{thm_inv_pse}. 



\subsection{Concave-convex composites: pseudoconvex example 2}

\begin{thm}\label{thm_comp}
Let $f:C\to\r$ be a function defined by 
\e{
f(x)=\varphi (g(x))
}
where $\varphi:I\to\r$ ($I\subseteq\r$ open) is concave, continuously differentiable, and monotonically increasing ($\varphi'(t)>0$ for any $t\in I$) and $g:C\to\r$ is convex and $g(C)\subseteq I$. Then, $f$ is invex with a kernel function 
\e{
\eta(x,y)=\frac{\varphi'(g(y))}{\varphi'(g(x))}(y-x).
}
\end{thm}
\begin{proof}
By the Chain rule of Clarke subdifferentials \cite[Theorem 2.3.9 (ii)]{clarke90}, $\partial f(x)=\varphi'(g(x))\partial g(x)$. We obtain
\e{\ald{
f(y)-f(x)
&= \varphi(g(y))-\varphi(g(x))\\
&\ge \varphi'(g(y))(g(y)-g(x))\\
&\ge \varphi'(g(y))\ang{\xi_g,y-x}\\
&= \ang{\xi_f,\frac{\varphi'(g(y))}{\varphi'(g(x))}(y-x)}
}}
for any $\xi_f=\varphi'(g(x))\xi_g\in\partial f(x)$ with $\xi_g\in\partial g(x)$, where the first inequality follows from the concavity of $\varphi$ and the second inequality follows from the convexity of $g$ and the positivity of $\varphi'(g(y))$.
\end{proof}

By Theorem \ref{thm_inv_pse}, the function defined in Theorem \ref{thm_comp} is actually pseudoconvex. We obtain the following three corollaries, which are used in the following subsection.

\begin{crl}\label{crl_log}
Let $f:C\to\r$ be a function defined by 
\e{
f(x)=\log g(x)
}
where $g:C\to\r$ is convex and $g(x)>0$ for any $x\in C$. Then, $f$ is invex with a kernel function 
\e{
\eta(x,y)=\frac{g(x)}{g(y)}(y-x).
}
\end{crl}

\begin{crl}\label{crl_p}
Let $f:C\to\r$ be a function defined by
\e{
f(x)=g(x)^p
}
where $0<p<1$ and $g:C\to\r$ is convex and $g(x)>0$ for any $x\in C$. Then, $f$ is invex with a kernel function 
\e{
\eta(x,y)=\prn{\frac{g(y)}{g(x)}}^{p-1}(y-x).
}
\end{crl}

\begin{crl}\label{crl_frac}
Let $f:C\to\r$ be a function defined by
\e{
f(x)=\frac{g(x)}{g(x)+c}
}
where $c>0$ and $g:C\to\r$ is convex and $g(x)>0$ for any $x\in C$. Then, $f$ is invex with a kernel function 
\e{
\eta(x,y)=\prn{\frac{g(x)+c}{g(y)+c}}^2(y-x).
}
\end{crl}
\begin{proof}
Set $\varphi(t)=t/(t+c)$ and we have $\varphi'(t)=c/(t+c)^2>0$ and $\varphi''(t)=-2c/(t+c)^3<0$ for any $t>0$ (hence $\varphi$ is concave).
\end{proof}

Kernel functions in Theorem \ref{thm_frac} and Corollary \ref{crl_log} have the same structure. Thus, due to Proposition \ref{prp_sum}, the invexity can be preserved by adding $\log h(x)$ and $g(x)/h(x)$, where $h$ is affine (convex and concave) and positive on $C$. The sum is again pseudoconvex due to Theorem \ref{thm_inv_pse}. This is not trivial since pseudoconvexity is generally not preserved under addition.

\begin{exm}\label{exm_frac_log}
Since $|x-1|/x$ and $\log x$ have a common kernel function $\eta(x,y)=(x/y)(y-x)$, $f(x)=|x-1|/x+\log x,\ x>0$ is invex (pseudoconvex) with the same kernel function.
\end{exm}

\subsection{Separable sums: non-pseudoconvex example 1}\label{sec_non_quasi}

As shown in Section \ref{sec_rel}, pseudoconvex functions are a subclass of quasiconvex functions, and thus, they always have convex sublevel sets. In contrast, invex functions can have nonconvex sublevel sets. We construct invex functions that are not pseudoconvex by a separable sum. Such examples are important in signal processing \cite{pinilla22,pinilla24}. A sum of pseudoconvex functions is not necessarily pseudoconvex, but the invexity is preserved if the sum is separable. The following result is also shown in \cite{pinilla24}.

\begin{thm}\label{thm_sep}
Let $f:\rn\to\r$ be the function defined by 
\e{
f(x)=\sum_{i=1}^n f_i(x_i).
}
If $f_i:\r\to\r$ is invex with a kernel function $\eta_i:\r\to\r$, then $f$ is invex with a kernel function
\e{
\eta(x,y)=[\eta_i(x_i,y_i)]_{i=1}^n
}
where $[a_i]_{i=1}^n\in\rn$ is the vector whose $i$-th component is $a_i\in\r$.
\end{thm}
\begin{proof}
$\partial f(x)=\{[\xi_i]_{i=1}^n \mid \xi_i\in\partial f_i(x_i)\}$ directly follows from Definition \ref{dfn_cla}. We obtain
\e{
f(y)-f(x)= \sum_{i=1}^n (f_i(y_i)-f_i(x_i)) \ge \sum_{i=1}^n \xi_i \eta_i(x_i,y_i) = \ang{\xi,\eta(x,y)}
}
for any $\xi=[\xi_i]_{i=1}^n\in\partial f(x)$ with $\xi_i\in\partial f_i(x)$.
\end{proof}

The following three examples are known as invex regularizers in signal processing \cite{pinilla22,pinilla24}. The invexity is proved by directly checking that every stationary point is globally optimal in \cite{pinilla22}. However, it can be proved in a simpler and more structured way by using Theorem \ref{thm_sep} and Corollaries \ref{crl_log}-\ref{crl_frac}.

\begin{exm}\label{ex_reg}
The following functions $f:\rn\to\r$ are invex:
\ite{
\item $f(x)=\sum_{i=1}^n\log(|x_i|+1)$ is invex with a kernel function
\e{
\eta(x,y)=\squ{\frac{1+|x_i|}{1+|y_i|}(y_i-x_i)}_{i=1}^n.
}
\item $f(x)=\sum_{i=1}^n(|x_i|+\epsilon)^p\ \ (0<p<1, \epsilon>0)$ is invex with a kernel function
\e{
\eta(x,y)=\squ{\prn{\frac{|y_i|+\epsilon}{|x_i|+\epsilon}}^{p-1}(y_i-x_i)}_{i=1}^n.
}
\item $f(x)=\sum_{i=1}^n\frac{|x_i|}{|x_i|+1}$ is invex with a kernel function
\e{
\eta(x,y)=\squ{\prn{\frac{|x_i|+1}{|y_i|+1}}^2(y_i-x_i)}_{i=1}^n.
}
}
\end{exm}

The graph of the first function in Example \ref{ex_reg} is shown in Figure \ref{fig_ex}(b). We can see that it has nonconvex sublevel sets, and hence it is not pseudoconvex. The graphs of the other two examples look similar and have nonconvex sublevel sets. They are used as regularizers to increase the sparsity of a solution to an optimization problem in signal processing \cite{pinilla22,pinilla24}.

By Theorem \ref{thm_inv_pse} and \ref{thm_sep}, a separable sum of pseudoconvex functions has a kernel function $\eta(x,y)=D(x,y)(y-x)$ where $D(x,y)$ is a positive semidefinite diagonal matrix.

For practical applications, it is important to establish invexity of sums of loss functions and regularizers; see \cite{pinilla22}. Extending the present kernel-based approach to such problems is left for future work.

\subsection{Convex functions with nonlinear transformations: non-pseudoconvex example 2}\label{sec_inv}

The following result, due to Craven \cite{craven81}, generates invex functions from convex functions. It is generalized to include nonsmooth functions.

\begin{thm}\label{thm_nonlinear}
Let $g:\rn\to\r$ be a convex function and $\Phi:X\to\rn$ be differentiable with the Jacobian $\nabla\Phi\in\r^{n\times n}$ nonsingular on $X$. Then $f:X\to\r$ defined by
\e{
f(x)=g(\Phi(x))
}
is invex with a kernel function
\e{
\eta(x,y)=(\nabla\Phi(x))^{-1}(\Phi(y)-\Phi(x)).
}
\end{thm}
\begin{proof}
Since $\partial f(x)=\nabla\Phi(x)^\top \partial g(x)$ by the chain rule of Clarke subdifferentials \cite[Theorem 2.3.9 (iii)]{clarke90}, we obtain
\e{\ald{
f(y)-f(x)
& = g(\Phi(y))-g(\Phi(x)) \ge \ang{\xi,\Phi(y)-\Phi(x)}\\
& = \ang{\nabla\Phi(x)^\top \xi,\nabla\Phi(x)^{-1}(\Phi(y)-\Phi(x))},
}}
for any $\xi\in\partial g(x)$.
\end{proof}

Although the transformation $\Phi$ can destroy the convexity, it preserves the invexity. That is the reason for the term ``invex (invariant convex)'' \cite{craven81}. This kind of construction of invex functions is generalized to the so-called $(h,F)$-convexity (see \cite{mishra08}).

We give an example, which is invex but whose sublevel sets are nonconvex (hence not pseudoconvex).


\begin{exm}[Rosenbrock's banana function (see \cite{naser25})]

Define $g(u_1,u_2)=u_1^2+bu_2^2$ and $\Phi(x_1,x_2)=(a-x_1,x_2-x_1^2)^\top$ with $a\in\r$ and $b>0$. Then, we obtain Rosenbrock's banana function 
\e{
f(x_1,y_1)=(a-x_1)^2+b(x_2-x_1^2)^2,
}
which is a well-known nonconvex benchmark test function in optimization. Here, $g$ is convex and
\e{
\nabla \Phi(x_1,x_2)=
\begin{pmatrix}
-1 & 0\\
-2x_1 & 1
\end{pmatrix},\quad
\det(\nabla \Phi(x_1,x_2))=-1,
}
the Jacobian matrix $\nabla \Phi(x_1,x_2)$ is nonsingular for every $(x_1,x_2)\in \mathbb{R}^2$. Hence $f$ satisfies the assumptions of Theorem \ref{thm_nonlinear}, and therefore $f$ is invex. It is not quasiconvex since for $(a,b)=(1,100)$, $f(0,0)=f(2,4)=1$, but for $(1,2)=((0,0)+(2,4))/2$, we have $f(1,2)=100$.

Moreover,
\e{
\nabla f(x_1,x_2) =
\left(
-2(a-x_1)-4bx_1(x_2-x_1^2),
2b(x_2-x_1^2)
\right)^\top.
}
If $\nabla f(x_1,x_2)=0$, then the second component gives $x_2-x_1^2=0$. Substituting this into the first component, we obtain $x_1=a$. It then follows that $x_2=a^2$. Therefore, $(a,a^2)$ is the unique stationary point (and global minimizer) of $f$.

Due to Theorem \ref{thm_nonlinear}, its kernel function is
\e{\ald{
\eta(x,y)
&=\left(
y_1-x_1,y_2-x_2-(y_1-x_1)^2
\right)^\top\\
&=\begin{pmatrix}
1 & 0\\
x_1-y_1 & 1
\end{pmatrix}
(y-x).
}}
Thus, the kernel function can be written as $\eta(x,y)=A(x,y)(y-x)$. Also, by suitably choosing $g$ and $\Phi$, we can create other smooth and nonsmooth invex test functions in optimization.
\end{exm}



\section{Conclusion}

We presented explicit construction methods for kernel functions together with concrete examples of invex functions. We also clarified the relationships among nonsmooth invex, pseudoconvex, and quasiconvex functions, and established a characterization of pseudoconvexity in terms of kernel functions. These results provide a constructive framework for analyzing invex functions and are particularly useful for nonsmooth, non-pseudoconvex examples arising in applications.

Future work includes developing a more systematic theory of invex functions based on kernel constructions and identifying tractable subclasses. In particular, it is of interest to study invex functions generated by the sum of convex functions and regularizers, as well as extensions to non-Lipschitz functions such as $\ell_p$-pseudonorms ($0<p<1$). Another interesting direction is to classify invex functions through the structure of their kernel functions, for example, functions admitting kernels of the form $A(x,y)(y-x)$ with structured matrices $A(x,y)$, such as the diagonal, triangular, or positive semidefinite cases appearing in Sections \ref{sec_non_quasi} and \ref{sec_inv}.

\section*{Acknowledgement}

The author is grateful to Mr.~Keiya Sakabe at LMU Munich for helpful discussions. The work of the author is partially supported by JSPS KAKENHI JP25KJ0120. 

\section*{Declarations}

The author has no conflict of interest.

\section*{Data availability statement}

This study did not use any datasets.

\begin{appendices}

\section{Proofs of propositions in Section \ref{sec_pre}}\label{a_proof}

\begin{proof}[Proof of Proposition \ref{prp_invex_non}]
First, assume $f$ is invex. If $x^*$ is a stationary point, then by \eqref{invex_non} with $\xi=0$, we obtain $f(x)\ge f(x^*)$ for any $x\in X$. Next, assume that every stationary point is a global minimizer. Consider
\e{\label{non_min}
\xi_x\coloneqq\underset{\xi\in\partial f(x)}{\mathrm{arg\,min}}\,\|\xi\|^2.
}
Such $\xi_x$ exists since $\partial f(x)$ is compact, and $\|\xi_x\|\neq0$ for any non-stationary $x$. Moreover, we have
\e{\label{non_opt}
\ang{\xi_x,\xi}\ge\|\xi_x\|^2,\ \ \forall \xi\in\partial f(x)
}
by the necessary optimality condition $\ang{2\xi_x,\xi_x-\xi}\ge0,\ \forall\xi\in\partial f(x)$ of the minimization problem \eqref{non_min}. We define $\eta$ by
\e{\label{eta_non}
\eta(x,y)\coloneqq
\begin{cases}
0 & \mathrm{if}\ 0\in\partial f(x),\\
\frac{-|f(y)-f(x)|\xi_x}{\|\xi_x\|^2} & \mathrm{if}\ 0\notin\partial f(x).
\end{cases}
}
Then, \eqref{invex_non} holds if $x$ is stationary. If $x$ is not stationary, we obtain 
\e{
f(y)-f(x) \ge-|f(y)-f(x)| \ge-|f(y)-f(x)|\frac{\ang{\xi_x,\xi}}{\|\xi_x\|^2}=\ang{\xi,\eta(x,y)}
}
for any $\xi\in\partial f(x)$ by \eqref{non_opt}. Thus, $f$ is invex with respect to $\eta$ defined above.
\end{proof}

\begin{proof}[Proof of Proposition \ref{prp_sum}]
We have
\e{\ald{
(\alpha f(y)+\beta g(y))-(\alpha f(x)+\beta g(x))
&\ge \alpha\ang{\xi_f,\eta(x,y)}+\beta\ang{\xi_g,\eta(x,y)}\\
&=\ang{\alpha\xi_f+\beta\xi_g,\eta(x,y)}
}}
for any $\xi_f\in\partial f(x)$ and $\xi_g\in\partial g(x)$. Since $\partial(\alpha f+\beta g)(x)\subseteq\alpha\partial f(x)+\beta\partial g(x)$ \cite[Proposition 2.3.3]{clarke90}, we conclude that $\alpha f+\beta g$ is invex.
\end{proof}

\begin{proof}[Proof of Proposition \ref{prp_const}]
Let $x^*\in X$ satisfy the KKT conditions \eqref{kkt}. Then, there exist $\xi_0^*\in\partial f(x^*)$, $\xi_i^*\in\partial g_i(x^*)$, and $\lambda_i^*\ge0$ such that $\xi_0^*=-\sum_{i=1}^m\lambda_i^*\xi_i^*$. By the invexity of $f$ and $g_i$ and $\lambda_i^*g_i(x^*)=0$, we obtain
\e{\ald{
f(x)-f(x^*)
&\ge\ang{\xi_0^*,\eta(x^*,x)}=-\sum_{i=1}^m\lambda_i^*\ang{\xi_i^*,\eta(x^*,x)}\\
&\ge -\sum_{i=1}^m\lambda_i^*(g_i(x)-g_i(x^*))=-\sum_{i=1}^m\lambda_i^*g_i(x)\ge 0
}}
for any $x\in X$ satisfying $g_i(x)\le 0$ for all $i$.
\end{proof}

\end{appendices}

\bibliography{ref}%

\end{document}